\documentclass{amsart}

\usepackage{latexsym,amsmath,amsfonts,amscd,amssymb,graphicx}
\allowdisplaybreaks

\theoremstyle{plain}
\newtheorem{theorem}{Theorem}[section]
\newtheorem{lemma}{Lemma}[section]
\newtheorem{proposition}{Proposition}[section] 
\newtheorem{corollary}{Corollary}[section]

\theoremstyle{definition}
\newtheorem{definition}{Definition}[section]
\newtheorem{example}{Example}[section]

\newtheorem{remark}{Remark}[section]

\begin{document}

\title{Invariant polynomials, gaps, and sparseness}

\author{John P. D'Angelo}

\address{Dept. of Mathematics, Univ. of Illinois, 1409 W. Green St., Urbana IL 61801}

\email{jpda@illinois.edu}

\author{Dusty E. Grundmeier}

\address{Dept. of Mathematics, The Ohio State University, 231 West 18th Avenue,
Columbus, Ohio 43210 }

\email{grundmeier.1@osu.edu}

\author{Daniel A. Lichtblau}

\address{Wolfram Research, Champaign IL 61802}

\email{danl@wolfram.com}

\maketitle

\begin{abstract} We consider each of the three classes of representations of cyclic groups that arise in the study of rational sphere maps.
We study the possible number of terms for invariant polynomials with non-negative coefficients that are constant on the appropriate line or hyperplane.
Our result provides crucial information about gaps in the possible target dimensions for certain invariant polynomial sphere maps.
We interpret our results in terms of sparseness for solutions of certain linear systems.

\medskip

\noindent {\bf AMS Classification Numbers}:  32H35, 32M99, 32V99, 11B13, 11B05.

\medskip

\noindent {\bf Key words}: Rational sphere maps, Hermitian invariant groups, gaps, sparse linear systems, Frobenius number,  degree estimates.

\end{abstract}

\section{Introduction}

Let $\Gamma$ denote a fixed-point free finite subgroup of the unitary group ${\bf U}(n)$. Two of the authors proved (see \cite{L} and \cite{DL}) that the existence of a non-constant rational 
sphere map invariant under $\Gamma$ restricts $\Gamma$ considerably. The group must be cyclic and represented in one of only three ways. 
Each of these three situations leads to questions about the possible minimal target dimensions for such maps. These questions lead us to study those invariant polynomials
with non-negative coefficients that are constant on a line or plane and ask how many terms such polynomials can have.
In our situation the set of possible values (of the number of terms) satisfies a Frobenius or postage stamp property; see Proposition 2.3. 
Not every value is possible. Thus there are {\bf gaps} in the target dimensions for (so-called minimal) invariant maps. 
Gaps in the possible target dimensions for rational sphere maps also exist. See \cite{HJX} and \cite{HJY} for various results about such gaps.  Because of invariance, the combinatorial results
in this paper differ considerably. The invariance also has consequences for sparseness of linear systems.

The three possible groups are described as follows. 

\begin{example} Let $\eta$ be a primitive $m$-th root of unity. Let $\Gamma(m,1)$ be the cyclic group of order $m$ generated by  $\eta  I$, where $I$ denotes the identity map on complex Euclidean space ${\mathbb C}^n$
for some $n$. 
 \end{example}
 
 \begin{example} Let $\eta$ be a primitive $p$-th root of unity where $p$ is odd.   Let
$ \Gamma(p,2)$ be the cyclic group of order $p$ generated by $\eta I_1 \oplus \eta^2 I_2$, where $I_1$ and $I_2$ denote  identity matrices on complex Euclidean spaces of arbitrary dimension. \end{example} 

\begin{example} Let $\eta$ be a primitive  $7$-th root of unity. Let $\Gamma$ be the cyclic group of order $7$ generated by 
$\eta I_1 \oplus \eta^2 I_2 \oplus \eta^4 I_3$. Here each $I_k$ denotes the identity matrix in an arbitrary dimension (not necessarily equal).\end{example} 

The most interesting case is perhaps Example 1.2, but Example 1.3 involves a new phenomenon we discuss in Section 5. 
In each case we will assume that the source dimension is minimal. Hence in 
Example 1.1 the source dimension will equal one, and this situation is quite simple.
In Example 1.2 the minimum source dimension is $2$, and in Example 1.3 the minimum source dimension is $3$. 
If we allow the identity matrices in the above examples to be in larger dimensions, the combinatorics becomes more difficult.
The natural source dimensions in the three cases are $1,2,3$ respectively.

{\bf Convention}. We henceforth use the following notation: $\Gamma(m,1)$ denotes the group in Example 1.1 where the source dimension equals $1$ and $\Gamma(2r+1,2)$ denotes the group in Example 1.2
where the source dimension equals $2$. Finally $\Gamma(7)$ denotes the group in Example 1.3 where the source dimension equals $3$. We discuss the group from Example 1.1 in source dimension $2$ in Section 4,
where we explain and use the notation $\Gamma_2(m,1)$. We discuss $\Gamma(7)$ in Section 5.

\begin{remark} The groups $\Gamma(p,q)$, defined as in Example 1.2 in source dimension $2$, but with the exponent $2$ replaced by an exponent $q$ relatively prime to $p$,  fit to some extent in this story.
They arise for sphere maps only when $q$ can be chosen to equal $1$ or $2$. The condition (required for sphere mappings) that the coefficients be non-negative will thus be a crucial assumption. \end{remark}

For each of the above groups we are therefore interested in real polynomials satisfying the following three properties:

\begin{definition} Let $\Gamma$ denote any of the groups from Examples 1.1, 1.2, and 1.3. A real polynomial is $\Gamma$-{\bf special} if 
\begin{itemize}
\item $f$ is $\Gamma$-invariant. In the first case $f(\eta x) = f(x)$ in source dimension $1$ and $f(\eta x,\eta y) = f(x,y)$ in source dimension $2$. In the second case $f(\eta x, \eta^2y) = f(x,y)$. In the third case,
$f(\eta x, \eta^2 y, \eta^4 z) = f(x,y,z)$. In each case $\eta$ is a root of unity as in Examples 1.1,1.2, and 1.3.
\item In the first case, $f(1)=1$. In the second case, $f(x,y)=1$ on the line $x+y=1$. In the third case, $f(x,y,z)= 1$ on the plane given by $x+y+z=1$.
\item The coefficients of $f$ are non-negative, and its constant term is $0$.
\end{itemize}
We say that $f$ is a {\bf good} polynomial if it satisfies the second and third items. Throughout this paper, the phrase ``has  non-negative coefficients" means ``has {\bf only} non-negative coefficients."
\end{definition} 

When the group is understood we just say {\it special}. Our main result determines  precise information on the gaps in the possible number of terms in special polynomials. See Theorem 1.2.


The following general uniqueness result holds for polynomials invariant under the group $\Gamma(p,q)$, 
but (unless $q=1$ or $q=2$) the coefficients must be of both signs. It is a special case of a result from 1992 appearing in \cite{DL}.

\begin{theorem} Put $\Gamma = \Gamma(p,q)$. There is a unique polynomial  $\Phi(x,y)=\Phi_\Gamma(x,y)$ satisfying the following properties:
\begin{itemize}
\item $\Phi$ is $\Gamma$-invariant. Thus $\Phi(\eta x, \eta^q y) = \Phi(x,y)$. 
\item $\Phi(x,y)=1$ on the line given by $x+y=1$.
\item $\Phi(0,0) = 0$.
\item $\Phi$ is of degree $p$.
\end{itemize}
Furthermore $\Phi$ is given by the following formula:
\begin{equation}  \label{eq:prod}
\Phi(x,y) = 1 - \prod_{j=1}^p (1 - \eta^j x - \eta^{qj} y). 
 \end{equation}
\end{theorem}

\begin{remark} The constant polynomial $1$ satisfies the first two properties, but it is not important in our discussion. It is precluded by our assumption that $f(0,0)=0$.  This assumption
is required for uniqueness results. \end{remark} 

\begin{remark} By \cite{DL}, a general uniqueness result holds for polynomials invariant under any fixed-point free subgroup $\Gamma$ of ${\bf U}(n)$, and a formula analogous to \eqref{eq:prod} defines $\Phi_\Gamma$.
Except in the three cases noted, however, the coefficients of the corresponding real polynomials must be of both signs. Allowing coefficients of both signs arises when we map spheres to hyperquadrics,
but will not be a major concern in this paper. For general $q$, combinatorial formulas for the integers arising in \eqref{eq:prod}  appear in \cite{LWW}. See also \cite{G}, \cite{GLW}, and \cite{G2} for additional information in this case. \end{remark}

\begin{remark} For the groups $\Gamma(m,1)$ and $\Gamma(2r+1,2)$, the polynomial $\Phi_\Gamma$ includes all the non-constant invariant terms that generate the algebra of invariant polynomials.
For the group $\Gamma(7)$, the polynomial $\Phi_\Gamma$ includes all of these terms and also several of their products. In general, for $\Gamma(p,q)$,  not all the monomials generating the algebra appear
in $\Phi_\Gamma$.  See \cite{G2}.
\end{remark} 

We call $\Phi_\Gamma$ the {\bf basic} polynomial for $\Gamma$. When $\Gamma= \Gamma(m,1)$, the basic polynomial is given by $x^m$ in source dimension $1$ and by $(x+y)^m$ in source dimension $2$.
Applying Theorem 1.1 to $\Gamma(2r+1,2)$ yields Corollary 1.1 below. Reference \cite{Drat} includes an entire chapter on these polynomials, 
including proofs of the various statements made in the subsequent two paragraphs. See also \cite{D2} and \cite{D3} for further results. The basic polynomial for $\Gamma(7)$ appears in Proposition 1.1.

We recall some information about the group $\Gamma(2r+1,2)$ in the following corollary of Theorem 1.1.

\begin{corollary} For each $r$ there is a unique polynomial $f(x,y)=f_{2r+1}(x,y)$ satisfying the following properties:
\begin{itemize}
\item $f$ is $\Gamma(2r+1,2)$-invariant. Thus $f(\eta x, \eta^2y) = f(x,y)$.
\item $f(x,y)=1$ on the line $x+y=1$.
\item The coefficients of $f$ are non-negative and $f(0,0)=0$. 
\item $f$ is of degree $2r+1$.
\end{itemize}
The polynomial $f_{2r+1}(x,y)$ has the following explicit formulas:
\begin{subequations}
\renewcommand{\theequation}{\theparentequation.\arabic{equation}}
\begin{align}
f_{2r+1}(x,y) &= \left( {x + \sqrt{x^2 + 4y} \over 2}\right) ^{2r+1} + \left( {x - \sqrt{x^2 + 4y} \over 2}\right) ^{2r+1} + y^{2r+1} \label{eq:21}\\
f_{2r+1}(x,y) &= x^{2r+1} + y^{2r+1} + \sum_{j=1}^r c(r,j) x^{2r+1-2j} y^j. \label{eq:22}\\
\intertext{The coefficients $c(r,j)$ are positive integers that  satisfy}
&c(r,j) = {2r+1 \over j} {2r-j \choose j-1} . \label{eq:22}
\end{align}
\end{subequations}
\end{corollary}

This family of polynomials $f_{2r+1}$ satisfies many intriguing properties, of which we mention several here. The polynomial
$f_{2r+1}(x,y)$ is congruent to $x^{2r+1}+ y^{2r+1}$ mod $(2r+1)$ if and only if $2r+1$ is either prime or $1$. The family of polynomials satisfies a third-order recurrence relation and has an explicit rational generating function. 
See \cite{Drat} for more information.

The number of terms in a polynomial will be of particular importance. Given a polynomial $f$ (in any number of variables),
$N(f)$ denotes the number of distinct monomials arising in $f$.
By \cite{DKR}, the polynomials $f_{2r+1}$ satisfy the following optimization property, or sharp {\bf degree estimate}.
Consider the collection of good polynomials $f$ of degree $d(f)$. Thus $f(x,y) =1$ on the line $x+y=1$, the coefficients of $f$ are non-negative, and $f(0,0)=0$.
Then $d(f) \le 2N(f)-3$ and equality holds for the polynomial $f_{2r+1}$, where $N=r+2$.
The real polynomial $f_{2r+1}$ is the squared Euclidean norm of a holomorphic polynomial $p(z,w):{\mathbb C}^2 \to {\mathbb C}^{r+2}$  that maps the unit sphere in the source
 to the unit sphere in the target. 
The number of terms is the minimum embedding dimension of such a map, and thus the sharp degree estimate $d \le 2N-3$ holds for these polynomials.
By \cite{LP},  the sharp degree estimate $d \le {N-1 \over n-1}$ holds in dimensions at least $3$. These bounds therefore limit the number of unknown coefficients one must consider.


We mention the formula for $\Phi_\Gamma$ for the group $\Gamma(7)$, where the source dimension is $3$. Write the variables as $(x,y,z)$.
This formula appears in both \cite{C} and \cite{DL}.

\begin{proposition} The basic polynomial for the group $\Gamma(7)$ is given by
\begin{align} \Phi(x,y,z) &= x^7 + y^7 + z^7 + 14(x^3 y^2 + x^2 z^3 + y^3 z^2 + xyz)\notag \\
& + 7(x^5 y + x z^5 + y^5 z + x y^3 + x^3 z + y z^3) \label{eq:fchiappari}\\
& + 7( x y^2 z^4 + x^2 y^4 z + x^4 y z^2 + x^2 y^2 z^2) \notag
\end{align}
 \end{proposition}

The main result in this paper is the following theorem on the gaps for invariant polynomials. The first part of the theorem is easy. The second part answers a question from \cite {BG}. The third part is completely new.

\begin{theorem}
Assume $f$ is a $\Gamma$-special polynomial; thus it satisfies the properties from Definition 1.1 for the group $\Gamma$. The following statements hold and are sharp.
\begin{itemize}
\item For $\Gamma(m,1)$, the number of terms $N(f)$ is an arbitrary non-negative integer.
\item For $\Gamma(2r+1,2)$, the number of terms $N(f)$ can be either $r+2$ or any positive integer at least $2r+3$. No other values are possible. 
\item For $\Gamma(7)$, the situation is considerably more complicated. See Theorem 5.1.
\end{itemize}
 \end{theorem}
 
 In Section 6 we interpret these results in terms of sparse solutions to linear systems. Finding special polynomials amounts to solving a linear system for the coefficients,
 and finding the minimum number of possible non-zero coefficients seeks the sparsest solution. Gaps in the possible numbers of non-zero coefficients have a related interpretation.

In Section 7 we note several possible related directions of research.

The main new results in this paper are the proof of the second statement in Theorem 1.1 and almost all of Section 5, especially Theorem 5.1 and the methods used to obtain it.
The authors acknowledge Franc Forstneric for useful discussions from years ago, Xiaojun Huang for his work on related problems,  Jiri Lebl for his contributions to this area, 
and Ming Xiao for more recent conversations.

\section{Invariance and $N(f)$}

The following lemma about the number of terms in polynomials is obvious.

\begin{lemma} Suppose $F_1$ and $F_2$ are polynomials with no common terms. Then $N(F_1+F_2)= N(F_1) + N(F_2)$. In particular, if all the monomials in $F_2$ are of higher degree than those of $F_1$,
then $N(F_1+F_2) = N(F_1) + N(F_2)$. \end{lemma}

Let ${\bf x}$ denote either $(x,y)$ or $(x,y,z)$.
We will be considering polynomials that equal $1$ on the sets $x+y=1$ or $x+y+z=1$. In the first case such a polynomial must be divisible by $x+y-1$, and in the second case by $x+y+z-1$.

When the polynomial is invariant, it must be also divisible by the terms obtained by replacing the variable ${\bf x}$ by $\gamma {\bf x}$, for each $\gamma$ in the group. 
Hence we have the following result.

\begin{proposition} \label{prop:construct} Assume $G$ is special for $\Gamma$.  Then there is an invariant polynomial $H$ such that
$ G = \Phi_\Gamma - H + H \cdot \Phi_\Gamma$.     \end{proposition}

\begin{proof} Since $G - \Phi_\Gamma$ is invariant and vanishes where $\Phi_\Gamma=1$, it is divisible by $\Phi_\Gamma -1$. Since it is invariant, the quotient $H$ 
also is and thus $G - \Phi_\Gamma  = H \cdot (\Phi_\Gamma -1)$. \end{proof}

The operation of replacing $F$ with $F-H+FH$ is a special case of a partial tensor product operation introduced by the first author in his study of rational sphere maps. See \cite{Drat}. We will use this operation throughout.
The condition that our polynomials have non-negative coefficients requires care. We write $g \preceq h$ when all the coefficients of $h-g$ are non-negative. Suppose $0 \preceq F$.
Put $G= F-H + HF$. 
If $0 \preceq H \preceq F$, then obviously $0 \preceq G$. It is possible, however, that $0 \preceq G$, but that $0 \preceq F - H$ is false. 

\begin{example} Put $F=\Phi_\Gamma$ for any of the three groups. For each $k$, the polynomial $F^k$ is special. For $k\ge 3$, 
put $H= F+F^2 + \dots + F^{k-1}$. Then {\bf none} of the coefficients of $F-H$ are positive, and yet 
$$ F-H + HF = F-( F+F^2 + \dots +F^{k-1}) + (F^2 + F^3 + \dots + F^k) = F^k. $$ 
 \end{example}

We remind the reader that our primary purpose
is to determine precisely the possible values  of $N(G)$ for special $G$. We begin with two easy lemmas. 

\begin{lemma} Suppose $F=1$ on the line defined by $x+y=1$. Let $H$ be an arbitrary polynomial. Put $G = F - H + FH= F+H(F-1)$. Then $G=1$ on this line. If $F$ and $H$ are group invariant,
then so is $G$. Analogous  conclusions hold when $F=1$ on the plane $x+y+u=1$. \end{lemma}

\begin{lemma} Suppose $F$ has non-negative coefficients.  Let $H$ consist of some terms of $F$, with corresponding coefficients that are positive but strictly smaller than those of $F$.
Then $N(F-H) = N(F)$. \end{lemma}

\begin{proposition} Given a good polynomial $f$ and a non-negative integer $a$,  there are good polynomials $g$ and $h$ such that 
$$ N(g) = a N(f)$$
$$ N(h) = a N(f) - (a-1).$$
When $f$ is group-invariant we may choose  $g$ and $h$ to be invariant.
\end{proposition}
\begin{proof}
Given $f$, find a monomial $m$ in $f$ of largest degree. For $0 < \lambda < 1$, we form
$$ g = (f-\lambda m)  +\lambda mf, $$
$$ h = (f - m) + mf. $$
By the above lemmas, 
$N(g) = N(f) + N(f) = 2N(f)$ and $N(h) = 2N(f) -1$. Iterating each of these operations (always on a monomial of largest degree), we obtain the conclusion.
The invariance follows because $f$ is invariant if and only if each monomial in $f$  is invariant. \end{proof}

The proof yields the following useful result, which we call the Frobenius property or the postage stamp property.

\begin{proposition} Given special polynomials $f$ and $g$, and non-negative integers $a,b$, there is a special polynomial $H$ with $ N(H) = a N(f) + b N(g)$.
Furthermore, if $h$ any special polynomial, then there is a special polynomial $H$ with
\begin{equation} N(H) = N(h) + a N(f) + b N(g). \label{eq:5} \end{equation} \end{proposition}

\begin{proof} Choose a monomial $m$ in $h$ of highest degree, and (for $0 < t  < 1$), define $h^*$ by $h^*=  h - tm + tmf $. By Lemmas 2.2 and 2.3, $h^*$ is invariant, equals $1$ on the line $x+y=1$, and
has non-negative coefficients. Furthermore, $N(h^*)= N(h) + N(f)$. Iterating this process, each time on a monomial of highest degree, after $a$ steps 
we obtain a polynomial $h^{**}$ with $N(h^{**}) = N(h) + a N(f)$.
Now do the same using $g$ to obtain a special polynomial $H$ satisfying \eqref{eq:5}. The first statement follows from the second by using $h=f$ and replacing $a$ by $a-1$.
\end{proof}

\begin{proposition} Let $p(t)$ be a polynomial of the form
$$ p(t) = 1 + \sum_{j=1}^K c_j t^j $$
where each $c_j \ge 0$.  Suppose at least $k$ of the $c_j$ are positive. Put $F(t)=(1+t)^m$.
Then $N(Fp) \ge m+1+k$. \end{proposition}
\begin{proof} The product $Fp$ contains the $m+1$ monomials $t^j$ for $0 \le j \le m$.  Suppose $a_1,\dots, a_k$ are distinct indices with $c_{a_i} > 0$. Then $Fp$ also contains the monomials whose exponents are
$m+a_1,\dots, m+a_k$. Hence there are at least $m+1+k$ terms.
 \end{proof}
 
 The product can contain a larger number of monomials, but no better estimate is possible. If $p(t) = (1+t)^k$, for example, then $N(Fp) = N((1+t)^{m+k}) = m+k+1$.
 
 \begin{example} Put $F(t) = (1+t)^m$ for $m\ge 2$.  Consider $p(t) = 1 + t + t^{m+1}$. Then $k=2$ but $N(Fp) = 2m+2 >m+3$. \end{example}
 
 \begin{remark} The proof of Theorem 1.2 is easy for $\Gamma(m,1)$ in source dimension $1$. We pause
 to make this observation now. Given an integer $m$ at least $2$, let $\eta$ be a primitive $m$-th root of unity. For $\lambda_j > 0$ and $\sum_{j=1}^d \lambda_j =1$, put
 $$ p(x) = \sum_{j=1}^d \lambda_j x^{mj}. $$
 Then $p(\eta x) = p(x)$, its coefficients are non-negative, $p(1)=1$, and $p$ has no constant term. Thus $p$ is {\bf special} for $\Gamma(m,1)$ in source dimension one, and
 $N(p)=d$.   Since $d$ is an arbitrary positive integer, there are no restrictions on $N(p)$ for $p$ special.  \end{remark}

We next describe the invariant monomials for each of our three classes of groups.
In any source dimension, the algebra of non-constant monomials invariant under $\Gamma(m,1)$ is generated by the constant $1$ and the homogeneous polynomials of degree $m$.
The degree of an invariant polynomial  with $f({\bf 0})=0$ must be a multiple of $m$.

The algebra of polynomials invariant under $\Gamma(2r+1,2)$ is generated by $1$, the monomials  $x^{2r+1}$, $y^{2r+1}$, and 
for $1 \le j \le r$, the monomials $x^{2r+1 -2j} y^j$.

For both of these classes of groups, the terms of the basic polynomial $\Phi_\Gamma$ are the basis elements of the algebra. That no longer holds for the group $\Gamma(7)$.
Several products are required. For the third group, with variables $x,y,z$, the algebra of invariant polynomials is generated by the following:
$$ x^7, y^7, z^7, x^5 y, x^3 y^2, xy^3, x^3 z, y^5 z, y^3 z^2, y z^3, xyz.  $$
But the basic polynomial from Proposition 1.1 has $17$ terms, including for example $x^2y^2z^2$ and $x^4 y z$. This situation, which was mentioned in Remark 1.4, changes the behavior for the gaps in possible values for $N(f)$.

For $\Gamma(2r+1, 2)$ the next easy result follows by combining the lemmas above.

\begin{proposition} Fix $r$. Let $F= f_{2r+1}$  denote the basic polynomial defined by Theorem 1.1. Formula \eqref{eq:22}, for certain positive integers $c(r,j)$, writes $F$ as:
$$ F(x,y) = x^{2r+1} + y^{2r+1} + \sum_{j=1}^r c(r,j) x^{2r+1-2j} y^j .$$
Let $H$ be defined by $\sum_{j=1}^r  \lambda_j x^{2r+1-2j} y^j$ where $0 \le \lambda_j < c(r,j)$ for each $j$.
Put $G=F-H + FH$. Then $ N(G) = N(F) + N(FH)$.
\end{proposition}
\begin{proof} Note that all terms in $FH$ are of higher degree than those of $F$. Thus  Lemma 2.1 implies $N(G) = N(F-H) + N(FH) = N(F) + N(FH)$. \end{proof}

We end this section with a warning. When the coefficients of polynomials $g,h$ can be of both signs, there is no simple way to relate $N(gh)$ to $N(g)$ and $N(h)$.

\begin{example} We give several examples.
\begin{itemize}
\item Put $g(x) = x^2 + 1 + \sqrt{2}x$ and $h(x) = x^2 +1 - \sqrt{2}x$. Then $N(g)=N(h)=3$, but $N(gh)=2$. 
\item Put $g(x,y) = x-y$ and $h(x,y) = x^m + x^{m-1}y + \dots + y^m$. Then their product $x^{m+1} - y^{m+1}$ has only two terms.
\item There are polynomials $p$ in one variable such that $p^2$ has fewer terms than $p$. See \cite{CD}. (Homogenizing then gives examples of homogeneous polynomials
in two variables with this property.)
\end{itemize} 
\end{example}

\section{Proof of first two parts of Theorem 1.2}
\begin{proof} In Remark 2.1 we have already observed the result for $\Gamma(m,1)$ in source dimension $1$.
We next give the proof for $\Gamma(2r+1,2)$. Let $F$ denote the basic map $\Phi_\Gamma = f_{2r+1}$. Recall that $N(F)=r+2$ and that
$$ F(x,y) = x^{2r+1} + y^{2r+1} + \sum_{j=1}^r c(r,j) x^{2r+1-2j} y^j .$$
We  first show that there is a special polynomial $G$ with $2r+3$ terms. In fact there are $r+2$ such examples.

We choose precisely one of the terms $c(r,j)x^{2r+1-2j}y^j $ and call it $H$. Thus $G= F-H+FH$, where
$$ G(x,y) = F(x,y) - c(r,j)  x^{2r+1-2j} y^j + \big(F(x,y) \big) c(r,j) x^{2r+1-2j} y^j. $$

We have multiplied a term of degree $2r+1-j$ by $F$, and hence $d(G)$ satisfies 
$$ d(G) = 2r+1-j  + (2r+1)= 4r+2-j \le 4r+1. $$
Since $F(x,y)=1$ on $x+y=1$, so does $G$. Each monomial occurring in $G$ is invariant. 
The non-negativity of the coefficients is clear; since the coefficient
of $x^{2r+1-2j} y^j$ in $F$ is precisely $c(r,j)$, we have eliminated one term and replaced it by $r+2$ terms with positive coefficients. We claim the resulting terms are distinct. Except for the term $H y^{2r+1}$, all
the exponents of $x$ in each of the new terms is even, and the exponents of $x$ in the old terms are all odd. The term $H y^{2r+1}$ is of degree larger than $2r+1$ and hence
it is also new.  Thus $G$ has $(r+1) + (r+2) = 2r+3$ terms (the old plus the new). This construction works for each $j$ with $1 \le j \le r$ and yields a different $G$. Hence there are at least $r$ such polynomials. We could also use the monomials $x^{2r+1}$ and $y^{2r+1}$ for $H$ to obtain two more examples with $2r+3$ terms.

Next we show how to get $N=2r+4$. The idea is similar, but this time we choose a number $\lambda$ with $0 < \lambda < c(r,j)$ and form
$$ G(x,y) = F(x,y) - \lambda x^{2r+1-2j} y^j + \big(F(x,y) \big) \lambda x^{2r+1-2j} y^j. $$
For this polynomial $G$, we have $ N(G) = (r+2)+ (r+2)= 2r+4$.

To obtain higher values, we use $k-j+1$ consecutive terms to define $H$. Thus we choose $\lambda_s$ such that $0 < \lambda_s < c(r,s)$. Put $ G = F- H + FH, $
where
$$ H(x,y) = \lambda_j x^{2r+1 - 2j} y^j + \dots + \lambda_k x^{2r+1-2k} y^k. $$ 
Now $F-H$ has $r+2$ terms, and $HF$ has $(r+2+ k-j)$ additional terms.
Thus $N(G) = (r+2) + (r+2 + k-j) $. We get all possible values for $N(G)$ for $G$ of degree up to $4r+2$ in this way.
To obtain higher values, we use these invariant polynomials in the same manner.

 

Suppose that $N(G) \le 2r+2$. We wish to show that $G=F$. 
By the degree estimate $d \le 2N-3$ from \cite{DKR},  if $N(g)\le 2r+2$, then 
 $$ d(G) \le 2(2r+2) - 3 = 4r+1. $$
If $H$ contained a term of degree $2r+1$, then $HF$ and hence $G$ would have a term of degree $4r+2$.  But $d(G) \le 4r+1$, and thus $d(H) < 2r+1$.
Since $H$ is invariant, and $H$ contains neither $x^{2r+1}$ nor $y^{2r+1}$, 
\begin{equation} H(x,y) = \sum_{j=1}^r \lambda_j x^{2r+1-2j} y^j. \label{eq:9} \end{equation}
To preserve  non-negativity of the coefficients of $G$, we must have $0 \le \lambda_j \le c(r,j)$ for each $j$. Thus we have found the general special polynomial of at most this degree.

 As before, write $G = F-H + FH$.
If all the $\lambda_j$ in \eqref{eq:9} are $0$, we obtain $G=F$ and $N(G)=r+2$. Otherwise, find the largest  $j$ for which $\lambda_j > 0$. If $\lambda_j=c(r,j)$, then 
$N(G) \ge N(F) -1 + N(F)$. If $\lambda_j< c(r,j)$, then $N(G) \ge N(F) + N(F)$. In either case, we obtain $N(G) \ge 2(r+2) -1 = 2r+3$. Thus $N(G)$ cannot lie in the integer interval $[r+3,2r+2]$.
We have already shown that $N(G)$ must be at least $r+2$ and that this value is possible. We have shown that each value at least $2r+3$ is possible. Thus we have proved the theorem for $\Gamma(2r+1,2)$.

\end{proof}

\begin{example} 

Put $r=5$. Then $F = f_{2r+1}$ satisfies
\begin{equation} F(x,y) = x^{11} + y^{11} + 11 x^9 y + 44 x^7 y^2 + 77 x^5 y^3 + 55 x^3 y^4 + 11 x y^5. \label{eq:10} \end{equation}
Put $H= 44 x^7 y^2 + 77 x^5 y^3 + 55 x^3 y^4$ and put $G= F- {H \over 11} + {FH \over 11}$. Then $G$ equals:

$$ x^{11} + 11 x^9 y + 40 x^7 y^2 + 4 x^{18} y^2 + 70 x^5 y^3 + 51 x^{16 }y^3 + 50 x^3 y^4 + 258 x^{14} y^4 + 11 x y^5 + 671 x^{12} y^5 $$
$$  + 979 x^{10} y^6 + 814 x^8 y^7 + 352 x^6 y^8 + 55 x^4 y^9 + y^{11} + 
 4 x^7 y^{13} + 7 x^5 y^{14} + 5 x^3 y^{15}.  $$ 
  The polynomial $G$ has $18$ terms; $7$ come from $F$ and $11$ more arise from multiplying
 three consecutive terms by $F$, creating $7+4=11$ more terms.
 \end{example}

\bigskip

\section{The group $\Gamma(m,1)$ in source dimension $2$}

We consider the group $\Gamma(m,1)$ when the source dimension is $2$.

\begin{theorem} Suppose $G$ is special for the group $\Gamma(m,1)$ and $G$ is of degree $Km$. Then there are homogeneous polynomials $H_k$ such that
$$ G = F-H_m + FH_m - H_{2m} + \dots + FH_{(K-1)M} - H_{Km} + F H_{Km}.$$
 If the $H_{jm}$ have non-negative coefficients, 
then $N(G) \ge Km+1$. \end{theorem}

\begin{proof} 

The result holds when $K=1$ because the only special polynomial of degree $m$ is the basic polynomial $F$, namely $(x+y)^m$, which has $m+1$ terms. 
Otherwise, by Proposition 2.1, we write $ G = (F-H) + HF$ for some invariant $H$. Note that $F(0,0)=G(0,0)=0$ implies $H(0,0)=0$. Write
$$ H= H_m + H_{2m} + \dots + H_{(K-1)m}$$
in terms of its homogeneous expansion. The term $H_{Km}$ must vanish, because otherwise $G$ would be degree $(K+1)m$.
From $G=(F-H)+HF$, we obtain
$$ G = (F-H_m) + (FH_m - H_{2m}) + \dots + ( FH_{(K-2)m} - H_{(K-1)m} ) + FH_{(K-1)m}. $$
Each of the terms in parentheses is of different degree and hence $N(G)$ equals the following sum:
$$ N(F-H_m) + N(FH_m - H_{2m}) + \dots + N(FH_{(K-2)m} - H_{(K-1)m} ) + N(FH_{(K-1)m}). $$
Also each term has non-negative coefficients. 
We group the terms in this sum in three parts: the first term, the last term, and the middle terms. Thus $$N(G) = N(F-H_m) + \sum_{j=1}^{K-2} N(FH_{jm} - H_{(j+1)m} ) + N(FH_{(K-1)m}) $$
Let $a_j = N(H_{mj})$. Since  $H_{(K-1)m} \ne0$, Proposition 2.4 implies that 
$$ N(FH_{(K-1)m}) \ge m+ a_{m-1}. $$
For the first term, we have 
$$N(F-H_m) \ge N(F) - a_1 = m+1 - a_1.$$
For the other terms, again by Proposition 2.4, we have
$$  N(FH_{jm} - H_{(j+1)m}) \ge N(  FH_{jm}) - a_{j+1} \ge m+ a_j - a_{j+1}. $$
The sum telescopes. Adding all the inequalities together yields

$$ N(G) \ge (m+1) + m(K-2) + m = Km+1. $$ 
\end{proof}

The restrictive assumption that the $H_j$ have non-negative coefficients implies Theorem 1.2 for the group $\Gamma(m,1)$ in source dimension $2$.
First we give an example showing that it can happen that some $H_j$ have at least one negative coefficient and yet $G$ is special.

\begin{example} 
Put $H_2(x,y) = x^2 - x y + y^2$. Define $G(x,y)$ by
$$ G(x,y) = (x+y)^2 - (x^2 -xy + y^2) + (x+y)^2 (x^2 - xy + y^2) = x^4 + 3 x y + x^3 y + x y^3 + y^4. $$
Thus $G$ has non-negative coefficients, $G= F-H_2 + FH_2$, and $H_2$ has a negative coefficient.

\end{example}

We prove our result for $\Gamma(m,1)$ under the restrictive condition that the 
polynomials $H_{jm}$ are non-negative.

\begin{proof} Let $G$ be special for $\Gamma(m,1)$. Then $G$ is of degree $Km$ for some positive integer $K$. By Theorem 4.1,  $N(G)= Km+1$.
When $K=1$, the only map is $F$, and $N(F)=m+1$. If $K\ge 2$, then $N(G) \ge Km+1\geq 2m+1$; it follows that $N(G)$ is in neither the integer interval $[1,m]$ nor the interval $[m+2,2m]$.

It remains to show that all larger integers are possible, even allowing this restrictive assumption.
We construct examples with $N$ terms for $N \ge 2m+1$. To obtain an example where $N=2m+1$, we put $$ G(x,y) = (x+y)^m - x^m + x^m (x+y)^m = F - H + HF.$$
Then $N(G) = (m) + (m+1) = 2m+1$. To obtain an example where $N=2m+2$, we put for $0 < \lambda < 1$,
$$ G(x,y) = (x+y)^m - \lambda x^m + \lambda x^m (x+y)^m$$ and hence
$N(G) = (m+1) + (m+1) = 2m+2$. In general, let $H$ be the first $k$ terms in the expansion of $(x+y)^m$ and let $0 < \lambda < 1$.
Again put $ G = F -\lambda  H + \lambda FH$. Then $N(G) = (m+1) + (m+k)$.
This procedure realizes all $N$ up to $3m+2$. For any $G$ with $N$ terms, we may choose a monomial $h$ of highest degree in $G$,
and replace $G$ with $G - h + hG$, and achieve $N(G) + m+1$ terms. If we replace $G$ with $G-\lambda h + \lambda h G$, then we achieve $N(G)+m+2$ terms.
Thus we can achieve everything in the interval $2m+1 \le N \le 4m+3$.
Continuing in this way, or by using the Frobenius property,
 we obtain examples with $N$ terms for each $N$ at least $2m+1$. We conclude that $N(G)$ cannot satisfy $1 \le N(G) \le m$ nor $m+2 \le N(G) \le 2m$, but is otherwise arbitrary.

 \end{proof}
 
 \begin{remark} When one allows $H$ to have negative coefficients, interesting cancellations can occur. The second author has an example where $p=4$ and the target dimension $N=8$, so the bound in Theorem 4.1 does not hold in general.
  \end{remark}

\section{The group $\Gamma(7)$.}

The squared Euclidean norm of the minimal invariant sphere mapping for this group is a real polynomial $F$ analogous to the polynomials $(x+y)^m$ and $f_{2r+1}(x,y)$ we have considered earlier. 
By Proposition 1.1 the basic polynomial is given by
\begin{align} F(x,y,z) &= x^7 + y^7 + z^7 + 14(x^3 y^2 + x^2 z^3 + y^3 z^2 + xyz)\notag \\
& + 7(x^5 y + x z^5 + y^5 z + x y^3 + x^3 z + y z^3) \label{eq:14}\\
& + 7( x y^2 z^4 + x^2 y^4 z + x^4 y z^2 + x^2 y^2 z^2).  \notag
\end{align}

The invariance is evident: $F(\eta x, \eta^2 y, \eta^4 z) = F(x,y,z)$. The connection to the group $\Gamma(7,2)$ is also evident. More subtle is that monomials other than the generators of the algebra of invariant polynomials arise in $F$.
For example, $x^2 y^2 z^2$ is the square of the basic invariant monomial $xyz$. This fact impacts the number of terms in a subtle way. Note also that the restriction to each of the two dimensional subspaces defined by setting a coordinate equal to $0$ gives the map $f_7$,
but with a different choice of seventh root of unity. Thus we see where the coefficients of $7$ and $14$ arise. It is a bit surprising that these integers are also coefficients of the terms that are products of all three variables.
We note that $F$ has $6$ terms of degree seven, $4$ terms of degree six, $3$ terms of degree five, $3$ terms of degree four, and $1$ term of degree three.

Let $G:{\mathbb R}^3 \to {\mathbb R}$ be a $\Gamma$-invariant polynomial with non-negative coefficients.
Suppose $G(x,y,z) = 1$ on the plane defined by $x+y+z=1$ and $G(0,0,0)=0$. We are interested in the possible values for $N(G)$. Proposition \ref{prop:construct} provides a method to construct such $G$. Begin with $F$ and let $m$ be any collection of terms of $F$ such that
all the coefficients of both $m$ and $F-m$ are non-negative. Thus $0 \preceq m \preceq F$. 
Then the polynomial $F-m + mF$ satisfies all of our conditions. More generally, if $G$ satisfies all of our conditions, then
$G- p + pG$ also does as long as $0 \preceq p \preceq G$. All examples are constructed in this fashion. Things are subtle because a term
might not satisfy $p \preceq F$, but does satisfy $p \preceq F - m+mF$.

The mapping $F$ from \eqref{eq:14} consists of $17$ terms all of degree at most $7$. 
There can be no smaller number of terms because each invariant polynomial that equals $1$ on the plane must be divisible by $F-1$.
Thus $1\le N(G) \le 16$ is impossible. 

Put $H(x,y,z) =14 xyz$. Consider the polynomial
$G= F-H + HF$. We have multiplied each term in $F$ by the same monomial, and hence $N(HF) =17 $. But $F-H$ and $HF$ overlap because there are four monomials in $F$
of degree at most $4$, namely $xyz, x^3 z, x y^3, yz^3$. When these get multiplied by $xyz$, we get no new monomials. The monomial $xyz$ is missing from $G$. Therefore
$$ N(G) = (17-1) + (17-4) = 29.$$ 

We continue by listing maps $G$ for which $N(G)$ takes on all the known possible values up to $57$. In each case,
   the map $G$ has the form $G= F-H + HF$ where $H$ is some invariant polynomial. Rather than writing all the lengthy formulas for $G$, 
  we write the (much shorter) formulas for $H$.    Given $k \ge 58$, the postage stamp principle implies that examples for which $N(G)=k$ exist. 
  
   \begin {itemize}
   \item Put $H=0$. Then $G=F$ and $N (G) = 17 $.
       \item Put $H = 14 xyz$. Then $N (G) = 29 $.
          \item Put  $H = \lambda 14 xyz$ for $0 < \lambda < 1 $. Then $N(G)=30$. 
           \item Put $H =  7x^3 z$. Then $N (G) = 32 $.
           \item Put $H = \lambda 7 x^3 z$ for $0 < \lambda < 1 $. Then $N(G)=33$. 
   
        \item Put $H= \lambda z^7$ for $0 < \lambda < 1 $. Then $N (G) = 34$.
        \item Put $H = 7 xy^3 + 14 xyz$. Then $N (G) = 37 $.
      
           \item Put $H = 7 xy^3 + \lambda 14 xyz$ for $0 < \lambda < 1 $. Then $N(G)=38$.
           \item Put $H = \lambda (7 xy^3 + 14 xyz) $ for $0 < \lambda < 1 $. Then $N (G) = 39 $.
        \item Put $H = 14 xyz + 14 x^3 y^2 $. Then $N (G) = 40 $.
           \item Put $H = 14 xyz + 203 x^2 y^2 z^2$. Then $N(G) = 41$.
         
             \item Put $H = 7 xy^3 + 7 yz^3 $. Then $N (G) = 42 $.
                \item Put $H = 7 xy^3 + 14 xyz + 7 y z^3 $. Then $N (G) = 43 $.

      \item Put $H = 7 xy^3 + \lambda 14 xyz + 7 yz^3 $ for $0 < \lambda < 1 $. Then $N (G) = 44 $.

         \item Put $H = \lambda 7 xy^3 + \lambda 14 xyz + 7 x^3 z$ for $0 < \lambda < 1 $. Then $N (G) = 45 $.
         \item Put $H = \lambda (7 xy^3 + 14 xyz + 7 yz^3) $ for $0 < \lambda < 1 $.   Then $N(G)=46$. 
                  \item Put $H = 7 xy^3 + 7 x^3 z + 14 xyz + 7 yz^3 $. Then $N (G) = 47$.
            \item Put $H = 7 xy^3 + 7 x^3 z + \lambda 14 xyz + 7 yz^3 $ for $0 < \lambda < 1 $. Then $N (G) = 48 $.
         \item Put $H = \lambda 7 xy^3 + 7 x^3 z + \lambda 14 xyz + 
      7 yz^3 $ for $0 < \lambda < 1 $, Then $N (G) = 49 $.
      
         \item Put $H = \lambda 7 xy^3 + 
      7 x^3 z + \lambda 14 xyz + \lambda 7 yz^3 $ for $0 < \lambda < 1 $. Then $N(G) = 50$.
    
         \item Put $H = \lambda (7 xy^3 + 7 x^3 z + 14 xyz + 7 yz^3) $ for $0 < \lambda < 1 $. Then $N (G) = 51 $.
         \item Put $ H = 
    14 x^3 y^2 + 7 x y^3 + 7 x^3 z + 
      14 x y z + \lambda 14 x^2 z^3 $ for $0 < \lambda < 1 $. Then $N (G) = 52 $.
         \item Put $ H = 14 x^3 y^2 + 7 x y^3 + 7 x^3 z + \lambda 14 x y z + \lambda 14 x^2 z^3 $  for $0 < \lambda < 1 $. Then $N(G)=53$.          
         \item Put $H = \lambda 14 x^3 y^2 + 7 x y^3 + 
      7 x^3 z + \lambda 14 x y z + \lambda 14 x^2 z^3 $  for $0 < \lambda < 1 $. Then $N(G) = 54$. 
               \item Put $H = \lambda 14 x^3 y^2 + \lambda 7 x y^3 + 
     7 x^3 z + \lambda 14 x y z + \lambda 14 x^2 z^3 $ for $0 < \lambda < 1 $. Then $N (G) = 55 $.
     \item Put $H = \lambda (14 x^3 y^2 + 7 x y^3 + 7 x^3 z + 14 x y z + 
     14 x^2 z^3) $ for $0 < \lambda < 1 $. Then $N (G) = 56 $.
      \item Put $H =  {x^7\over 2} + 13 xyz + 182 x^2y^2z^2$. Then $N(G)=57$.
   \end {itemize}
   
   \begin{remark} The map for which $N(G)=41$ exhibits an interesting phenomenon. 
  In this example, $F - 203 x^2y^2z^2 + (203 x^2y^2 z^2)F$ has the term $-196 x^2 y^2 z^2$. But when we subtract the term $14xyz$ and add back in $14xyz F$, this negative coefficient goes away.
  The arithmetic comes from $7 - 203 = -196 = -(14)^2$. Hence all coefficients  of $G$ are non-negative. If we perform the operations in the opposite order,
  we never see a negative coefficient!  The map for which $N(G)=57$ also illustrates this phenomenon. 
  It occurs because the basic polynomial $F$ includes terms such as $x^2y^2z^2$ that are not basis elements of the algebra of invariant polynomials.\end{remark}

\begin{example} To illustrate non-uniqueness, we give a different way to achieve $N(G)=51$.
Put $ H = 14 x^3 y^2 + 7 x y^3 +  7 x^3 z + 14 x y z + 14 x^2 z^3$. Then $N(G)=51$.  The example above used fewer terms in the polynomial $H$.
\end{example}

Let $G$ be special for $\Gamma(7)$. We will show, unless $N(G)=17$, it must be true that $N(G)\ge 29$.
We recall the degree estimate from [21]; when $n\ge 3$, we have $d \le \frac{N-1}{n-1}$. Here $n=3$ and hence $N \ge 2d+1$. To show that $N(G) \ge 29$ for all special polynomials $G$, it suffices to prove the result 
for $G$ of degree at most $13$. 

By Proposition 1.1, the basic polynomial $F$ for the group $\Gamma(7)$ is given by \eqref{eq:14}. 
Therefore  there are no special polynomials of degree less than $7$. We will show that there are none of degree $8$ or $9$ and that there is a one-parameter family of degree $10$. We then go on to study degrees $11,12,13$.

Let $G:{\mathbb R}^3 \to {\mathbb R}$ be a $\Gamma(7)$-invariant special polynomial. By Proposition 2.1, 
$G = F + H (F-1)= F-H + HF$ where $H$ is a $\Gamma(7)$-invariant polynomial.  We first show the simple statement that the degree of $G$ must be at least $10$ unless $G=F$. 

\begin{lemma} Let $G$ be a $\Gamma(7)$-invariant special polynomial.  Then either $G=F$ or $\text{deg}(G) \ge 10$. \end{lemma} 
\begin{proof} Each non-constant invariant monomial\ invariant under $\Gamma(7)$ has degree at least $3$. Hence, if $H \ne 0$, we must have 
$$ \text{deg}(G) = \text{deg}(FH) = 7 + \text{deg}(H)\ge10. $$
\end{proof}

\begin{corollary} There is no invariant example $G$ with $18 \le N(G) \le 20$.   \end{corollary}
\begin{proof} By the lemma, either $G=F$ and $N(G) = 17$, or $N(G) \ge 2(10)+1 = 21$. \end{proof}

Next we find all examples of degree $10$. The only possibility is that $H= \lambda xyz$ for some $\lambda \ne 0$. Since all the coefficients of $G$ are positive, in fact $0 < \lambda \le 14$.
The resulting maps have the form
$$ G(x,y,z) = F(x,y,z) + \lambda xyz F(x,y, z) - \lambda xyz. $$
Then $N(G)$ is $29$ or $30$, depending on whether $\lambda =14$. 

\begin{remark} It is useful to think of the collection  ${\bf S}_{10}$ of all examples of degree at most $10$. We identify this set with the closed interval $[0,14]$ on the real line.
For larger degrees $d$, the set of examples ${\bf S}_d$ of degree at most $d$  will always be a compact convex subset of some ${\bf R}^m$. 
One must systematically study $N(G)$ for $G$ in the boundary of ${\bf S}_d$. When the degree is $10$, the dimension $m$ equals $1$ and the boundary consists of $\lambda=0$ and $\lambda=14$. When $\lambda = 0$, $N(G)=17$.
When $\lambda = 14$, $N(G)= 29$, and otherwise $N(G)=30$. \end{remark}

The authors regard degree $11$ as a particularly illustrative case; things are simple enough to be illuminating, but complicated enough to exhibit the key issues.
We describe all examples of degree $11$. From the $HF$ term, we must have $\text{deg}(H)\leq 4$ since $F$ has degree 7. 
Thus, we can give the general form of $H$ by taking all non-constant $\Gamma(7)$-invariant monomials of degree at most 4; we put 
$$H(x,y,z)= U xyz+ B yz^3 +C  x^3 z  + D x y^3 $$  The set ${\bf S}_{11}$ is a compact subset of ${\mathbb R}^4$. We list all the terms in $G$ in terms of the coefficients of $H$:
\begin{align*}
& z^7  && x^7  && y^7  &\\
& 14 y^3 z^2   && 14 x^2 z^3   && 14 x^3 y^2  &\\
& 7 y^5 z  && 7 x z^5  && 7 x^5 y  &\\
& (7-B) y z^3  && (7-C) x^3 z  && (7-D) x y^3  &\\
& 7 B y^2 z^6  && 7 C x^6 z^2  && 7 D x^2 y^6  &\\
& 14 B y^4 z^5  && 14 C x^5 z^4  && 14 D x^4 y^5  &\\
& 7 B y^6 z^4  && 7 C x^4 z^6  && 7 D x^6 y^4   &\\
& B y z^{10}  && C x^{10} z  && D x y^{10}  &\\
& B y^8 z^3  && C x^3 z^8  && D x^8 y^3  &\\
& (7 B+D) x y^3 z^7   && (B+7 C) x^7 y z^3   && (C+7 D) x^3 y^7 z   &\\ 
& (7 B+7 D) x^2 y^5 z^4   && (7 B+7 C)x^4 y^2 z^5  && (7 C+7 D)x^5 y^4 z^2  &\\
& (7+7 U+14 B)x y^2 z^4   && (7+7 U+14 C)x^4 y z^2   && (7+7 U+14 D)x^2 y^4 z   &\\
& (U+7 B) x y z^8  && (U+7 C) x^8 y z   && (U+7 D) x y^8 z  &\\
& (7 U+14 B) x^2 y z^6  && (7 U+14 C) x^6 y^2 z  && (7 U+14 D) x y^6 z^2  &\\
&(14 U+7 B+7 D) x y^4 z^3   && (14 U+7 B+7 C) x^3 y z^4  &&  (14 U+7 C+7 D) x^4 y^3 z &\\
& (7 U+7 B+7 D) x^2 y^3 z^5  && (7 U+7 B+7 C) x^5 y^2 z^3 && (7 U+7 C+7 D) x^3 y^5 z^2  &\\
& (14-U) x y z  && (14 U+7) x^2 y^2 z^2   && (7 U+14 B+14 C+14 D) x^3 y^3 z^3  &
\end{align*}

We have listed the terms in three columns. The last row is exceptional. In each other row, the three columns are isomorphic. One sends $x$ to $y$, $y$ to $z$, $z$ to $x$, and one sends $B$ to $C$, $C$ to $D$, and $D$ to $B$. This symmetry is a consequence of the group invariance, and something analogous holds in all degrees.

Note also that the coefficient $U$ plays a different role than $B,C,D$. The non-negativity of {\it all} the coefficients puts upper and lower bounds on $U,B,C,D$. 
All the coefficients of $G$ must be non-negative, and hence the lower bounds for $B,C,D$ are $0$. But $U$ can be negative.

We first discuss the case $U<0$. Each of $U+7B, U+7C,U+7D$ is a coefficient. If $U<0$, then we must have $B,C,D>0$. Of the coefficients, $B$ occurs alone twice, $7B$ alone twice, and $14B$ alone once. 
That gives $5$ positive coefficients. By the symmetry in $B,C,D$, we see that there are at least 15 positive coefficients of this type. The rows beginning with $7B+D$ and $7B+7D$ create $6$ more terms.
The first three rows generate $9$ more. Therefore the total number of positive coefficients  is at least $30$. The last row also has a positive coefficient. Thus $N(G) \ge 31$. In fact, at least three more terms must be positive.
The coefficients $U+7B$ and $7U+14B$ cannot both vanish, since $U < 0$. By symmetry the same holds for $C$ and $D$. We conclude finally that $N(G) \ge 34$.

Next we discuss the case $U> 0$: 
The five rows preceding the final row generate $15$ terms. The coefficients in the second and third columns in the last row are also positive. Thus there are at least $17$ terms in addition to the $9$ terms in the first three rows.
If $B=C= D=0$, the degree is $10$ but we are in degree $11$. Hence at least one of these is non-zero; say $B$. The five terms in which $B$ appears alone then guarantee at least $17+9+5=31$ terms. Thus $N(G)\ge 31$.

We're left with the case where $U=0$. If also $B=C=D=0$, then $H=0$ and $N(G)=17$. If none of $B,C,D$ are $0$, then there is a huge number of terms, as only four of the above  terms have a minus sign anywhere.
If one is $0$, say $D=0$, with $B,C \ne 0$, we see also a profusion of terms. Finally $D=C=0$ and $B>0$ yields at least $32$ terms.
To see this, set $U=C=D=0$ in the chart of terms. The first three rows create $9$ positive coefficients. The next rows create 
$$ 2 + 1+1+1+1+1 + 2+2+3+1+1 + 2 +2+3=23$$
additional terms, Thus $N(G)\ge 32$. In fact $32$ is possible by putting $B=7$. If $0<B<7$, then we get $33$ terms.

We have proved the following.
\begin{proposition} Let $G$ be a special polynomial for $\Gamma(7)$. If the degree of $G$ is $7$, then $N(G)=17$. There is no such $G$  of degree $8$ or $9$. If the degree of $G$ is $10$,
then then either $N(G)=29$ or $N(G)=30$. If the degree of $G$ is $11$, then $N(G)\ge 31$. \end{proposition}

As the degree of $H$ increases, more $\Gamma(7)$-invariant monomials are allowed. The coefficients of $G$ become more complicated and a more delicate analysis is required. From the structure of the coefficient expressions in $G$, the coefficient $U$ of $xyz$ in $H$ is the only coefficient that could be negative until we include the coefficient $V$ of $x^2y^2z^2$ in $H$. Thus, the case analysis for $\text{deg}(H)\geq 6$ is more involved than above. 
We now analyze the case where $H$ is of degree $5$. We have $7$ parameters. 
$$ H(x,y,z)= U xyz+ B yz^3 +C  x^3 z  + D x y^3  + R y^3 z^2   + S x^2 z^3  + T x^3 y^2. $$  
As before we list all the terms:
\begin{align*}
    & \hspace{-2cm} z^7   && x^7  && y^7  &\\
    & \hspace{-2cm} 7 y^5 z && 7 x z^5  && 7 x^5 y  &\\
    & \hspace{-2cm} (7-B) y z^3   && (7-C) x^3 z  && (7-D) x y^3  &\\
    & \hspace{-2cm} 7 B y^2 z^6   && 7 C x^6 z^2  &&  7 D x^2 y^6 &\\
    & \hspace{-2cm} B y z^{10}   && C x^{10} z  && D x y^{10}  &\\
    & \hspace{-2cm} (14-R) y^3 z^2   && (14-S) x^2 z^3  && (14-T) x^3 y^2  &\\
    & \hspace{-2cm} R y^3 z^9   && S x^9 z^3  && T x^3 y^9  &\\
    & \hspace{-2cm} R y^{10} z^2  && S x^2 z^{10}  && T x^{10} y^2  &\\
    & \hspace{-2cm} 7 R x y^5 z^6   && 7 S x^6 y z^5   && 7 T x^5 y^6 z  &\\
    &\hspace{-2cm}  (B+7 R) y^8 z^3   && (C+7 S) x^3 z^8  && (D+7 T) x^8 y^3  &\\
    &\hspace{-2cm}  (14 B+7 R) y^4 z^5   && (14 C+7 S) x^5 z^4  && (14 D+7 T) x^4 y^5  &\\
    & \hspace{-2cm} (7 B+14 R) y^6 z^4  && (7 C+14 S) x^4 z^6  && (7 D+14 T) x^6 y^4  &\\
    &\hspace{-2cm}  (7 B+D+7 R) x y^3 z^7   && (B+7 C+7 S) x^7 y z^3  && (C+7 D+7 T) x^3 y^7 z  &\\
    & \hspace{-2cm} (7 B+7 D+7 R+7 S) x^2 y^5 z^4   && (7 B+7 C+7 S+7 T) x^4 y^2 z^5  && (7 C+7 D+7 R+7 T) x^5 y^4 z^2  &\\
    & \hspace{-2cm} (7 R+S) x^2 y^7 z^3   && (7 S+T) x^3 y^2 z^7   && (7 T+R) x^7 y^3 z^2  &\\
    & \hspace{-2cm} (7+7 U+14 B) x y^2 z^4   && (7+7 U+14 C) x^4 y z^2  && (7+7 U+14 D) x^2 y^4 z  &\\
    &\hspace{-2cm}  (14 U+7 B+7 D+14 R) x y^4 z^3   && (14 U+7 B+7 C+14 S) x^3 y z^4  && (14 U+7 C+7 D+14 T) x^4 y^3 z  &\\
    &\hspace{-2cm}  (7 U+14 D+7 R) x y^6 z^2   && (7 U+14 B+7 S) x^2 y z^6  && (7 U+14 C+7 T) x^6 y^2 z  &\\
    & \hspace{-2cm} (U+7 B) x y z^8   && (U+7 C) x^8 y z  && (U+7 D) x y^8 z  &\\
    & \hspace{-2cm} (7 U+7 B+7 D+14 R+14 S) x^2 y^3 z^5   && (7 U+7 B+7 C+14 S+14 T) x^5 y^2 z^3  && (7 U+7 C+7 D+14 R+14 T) x^3 y^5 z^2  &\\
    & \hspace{-2cm}  (14-U) x y z  &&   &&  &\\
    & \hspace{-2cm}  (14 U+7) x^2 y^2 z^2  &&   &&   &\\
    &\hspace{-2cm}   (7 U+14 B+14 C+14 D+7 R+7 S+7 T) x^3 y^3 z^3 \hspace{-4cm} &&   &&   &\\
    &\hspace{-2cm}  (7 R+7 S+7 T) x^4 y^4 z^4 && && &\\
\end{align*}

We need to count the possible number of positive coefficients. As before we consider $U<0$, $U>0$, and $U=0$ separately. Since the degree of $H$ is five, at least one of $R,S,T$ is positive. 

Case $U<0$: There are $6$ terms with positive constant coefficients.
Next, since the coefficients $U+7B, U + 7C, U+7D$ appear, we must have $B,C,D > 0$. There are $2$ terms involving $B$ alone; replacing $B$ with $C$ and $D$ we get a total of $6$ terms (with positive coefficients) of this type.
The terms $B+7R$, $14B + 7R$, $ 7B + 14 R$, $7B+D+7R$,   $7B+7D+7R+7S$ each generates $3$ terms in this way. These give $15$ additional  positive coefficients. Since not all $R,S,T$ are $0$, we may suppose $R>0$. There are $3$ terms with $R$ alone, 
generating $3$ more positive coefficients.  The row starting $7R+S$ generates $2$ more. The coefficient  of $x^4 y^4 z^4$ in the last row must be positive. The term $14-U$ must be positive. We therefore have at least
$$ 6 + 6+15 +3+ 2+ 1 = 33 $$
terms with positive coefficients.

Case $U>0$: Assume $R>0$ but that $B,C,D =0$. Let $\mu = 2$ if $R=14$ and $\mu = 3$ if $0<R<14$. Let $\eta=1$ if $U=14$ and equal $0$ if $0<U< 14$. 
Going down the rows, we  get at least
$$   3+ 3 +   3 +  0 +  0  + \mu + 1+1+1+1+1+1+1+2 +3 +3 +3 +3 +3 + \eta +1 + 1 + 1= 35 + \mu + \eta \ge 37.$$
terms with positive coefficients.

Case $U=0$: In order that the degree is $12$, at least one of $R,S,T$ is positive. Assume  $R>0$. If $B>0$, but $C=D=0$, we proceed down the rows, getting
$$ 3+ 3+ 2+1+1+2 +1+1 + 1+ 1+1+1+ 2+ 3 +2 + 3+ 2+ 1 + 3 + 1 + 1 +1 + 1 = 38 $$ terms with positive coefficients.
We get even more such terms when $C$ and/or $D$ is positive. 
It remains to check the case where $U=B=C=D=0$. Let $\mu =2$ if $R=14$ and $\mu = 3$ if $0 < R < 14$. Going down the rows gives
$$ 3 + 3 + 3 + 0 + 0 + \mu + 1 + 1 +1 +1 +1 +1 + 1 + 2+ 2+3 + 1+1+0+2+1+1+1+1= 31 + \mu. $$
We get at least $33$ terms and thus have proved the following result.

\vspace{-0.18cm}
\begin{proposition} Let $G$ be a special polynomial for $\Gamma(7)$ of degree $12$. Then $N(G) \ge 33$. \end{proposition}

\vspace{-0.18cm}
Finally we must deal with the harder case when $G$ has degree $13$. Thus $H$ has degree $6$.
There are 11 (non-constant) $\Gamma(7)$-invariant monomials of degree at most 6. Thus we put \vspace{-0.1cm}
\begin{align}\label{eq:H}
H&=U xyz+Byz^3+Cx y^3 +Dx^3z+Ry^3 z^2+Sx^2z^3+T x^3y^2\\&+K  y^5z+ L x z^5+M x^5 y + V x^2 y^2 z^2. \notag
\end{align}
We note the symmetries in these parameters. The numbers $U$ and $V$ play special roles. 
The parameters $B,C,D$ behave symmetrically, the parameters  $R,S,T$ behave symmetrically, and $K,L,M$ also do.
We illustrate this symmetry now. 

The invariant monomials are listed in such a way that $B$ is the coefficient of the monomial of degree $4$ omitting $x$. Then permuting the variables by $x\to y \to z\to x$ requires permuting the parameters by $B \to C \to D \to B$.
The same holds for $R,S,T$ and $K,L,M$. This symmetry explains the table of terms; the second and third terms in each row are obtained from  the first through these permutations. On the next page, we give the terms of the general example of degree at most $13$. In the following list, it is convenient to write $\phi(a,b,c)= a+7b+14c$.
\newpage
\begin{align*}
& \hspace{-1cm} z^7 &&  x^7  && y^7 & &\\
& \hspace{-1cm}(7-B) y z^3 && (7-C) x^3 z  && (7-D) x y^3 & &\\
& \hspace{-1cm}7 B y^2 z^6 && 7 C x^6 z^2 && 7 D x^2 y^6 &\\
& \hspace{-1cm}B y z^{10} && C x^{10} z && D x y^{10} &\\
& \hspace{-1cm}(14-R) y^3 z^2 && (14-S) x^2 z^3 && (14-T) x^3 y^2 &\\
& \hspace{-1cm}R y^3 z^9 && S x^9 z^3 && T x^3 y^9 &\\
& \hspace{-1cm}(14 B+7 R) y^4 z^5 && (14 C+7 S) x^5 z^4 && (14 D+7 T) x^4 y^5 &\\
& \hspace{-1cm}(7 K+L) x y^7 z^5 && (7 L+M) x^5 y z^7 && (K+7 M) x^7 y^5 z  &\\
& \hspace{-1cm}(7 K+7 L+7 R) x y^5 z^6  && (7 L+7 M+7 S) x^6 y z^5 && (7 K+7 M+7 T) x^5 y^6 z &\\
& \hspace{-1cm}(7-K) y^5 z  && (7-L) x z^5 && (7-M) x^5 y &\\
& \hspace{-1cm}K y^5 z^8 && L x^8 z^5 && M x^5 y^8 &\\
& \hspace{-1cm}K y^{12} z && L x z^{12} && M x^{12} y &\\
& \hspace{-1cm}(7 B+7 K+14 R) y^6 z^4 && (7 C+7 L+14 S) x^4 z^6 && (7 D+7 M+14 T) x^6 y^4 &\\
& \hspace{-1cm}(7 K+R) y^{10} z^2 && (7 L+S) x^2 z^{10} && (7 M+T) x^{10} y^2 &\\
& \hspace{-1cm}(B+14 K+7 R) y^8 z^3 && (C+14 L+7 S) x^3 z^8 && (D+14 M+7 T) x^8 y^3 &\\
& \hspace{-1cm}(7 B+D+14 L+7 R) x y^3 z^7 && (B+7 C+14 M+7 S) x^7 y z^3 && (C+7 D+14 K+7 T) x^3 y^7 z &\\
& \hspace{-1cm}(7+7 U+14 B) x y^2 z^4 && (7+7 U+14 C) x^4 y z^2  && (7+7 U+14 D) x^2 y^4 z &\\
& \hspace{-1cm}(14 U+7 B+7 D+14 R) x y^4 z^3 && (14 U+7 B+7 C+14 S) x^3 y z^4 && (14 U+7 C+7 D+14 T) x^4 y^3 z  &\\
& \hspace{-1cm}(7 U+14 D+14 K+7 R) x y^6 z^2 && (7 U+14 B+14 L+7 S) x^2 y z^6 && (7 U+14 C+14 M+7 T) x^6 y^2 z &\\
& \hspace{-1cm}(U+7 D+7 K) x y^8 z  && (U+7 B+7 L) x y z^8  &&  (U+7 C+7 M) x^8 y z&\\
& \hspace{-1cm}(7 K+7 R+S+7 V ) x^2 y^7 z^3 && (7 L+7 S+T+7 V )x^3 y^2 z^7 && (7 M+7 T+R+7 V ),x^7 y^3 z^2 &\\
& \hspace{-1cm}\phi(0,B+D+R+S,K+V) x^2 y^5 z^4 \hspace{-0.25cm}&& \phi(0,C+B+S+T,L+V)x^4 y^2 z^5 \hspace{-0.25cm}&& \phi(0,C+D+R+T,M+V)x^5 y^4 z^2 \hspace{-0.25cm}&\\
& \hspace{-1cm}\phi(V,U+B+D+L,R+S)x^2 y^3 z^5 \hspace{-0.25cm}&& \phi(V,U+C+B+M,S+T)x^5 y^2 z^3 \hspace{-0.25cm}&& \phi(V,U+D+C+K,T+R)x^3 y^5 z^2  \hspace{-0.25cm}&\\
& \hspace{-1cm}(7 K+V )x^2 y^9 z^2 && (7 L+V )x^2 y^2 z^9  && (7 M+V )x^9 y^2 z^2  &\\
& \hspace{-1cm}(7 K+7 V )x^4 y^6 z^3  && (7 L+7 V )x^3 y^4 z^6 && (7 M+7 V )x^6 y^3 z^4  &\\
&\hspace{-1cm}(14-U) x y z  &&   &&   &\\
&\hspace{-1cm} (14 U-V +7)x^2 y^2 z^2 &&  &&  &\\
& \hspace{-1cm}(7 U+14 B+14 C+14 D+7 R+7 S+7 T+14 V )x^3 y^3 z^3 \hspace{-4cm} &&  &&  &\\
& \hspace{-1cm}(7 R+7 S+7 T+7 V )x^4 y^4 z^4.   &&  &&  &
\end{align*}


\vspace{-0.4cm}
\begin{theorem}
Let $G:{\mathbb R}^3 \to {\mathbb R}$ be a $\Gamma(7)$-invariant special polynomial. Then the number $N(G)$ of terms satisfies the following:
\begin{itemize}
\item  There is no $G$ with $0\leq N(G)<17$.
\item There is no $G$ with  $17<N(G)<29$. 
\item There are $G$ with $N(G) = 29,30, 32,33,34$ and $N(G) = k$ for each $k \ge 37$.
\item The authors believe that the numbers $31,35,36$ are impossible but have not proved this statement.
\end{itemize} 
\end{theorem}

\begin{proof}
Our list of examples include those of degree $17, 29,30, 33,34$ as well as all degrees from $37$ to $57$. By the postage stamp principle from Proposition 2.3, there are examples for all integers exceeding $57$.
Since $G(0)=0$ by definition of a special polynomial, degree $0$ is impossible. The minimum possible value of $N(G)$ occurs when $G$ is the basic polynomial  $F$ and therefore $0 \le N(G) < 17$ is impossible.

Let $G$ be a $\Gamma(7)$-invariant special polynomial. We write $G = F-H + HF$, as usual, where $H$ is some $\Gamma(7)$-invariant polynomial. Since $G$ is non-constant, we have $N(G)\neq 0$. 
When the degree of $G$ is at most $12$, the previous Propositions show that either $H=0$ and $N(G)=17$, or that $N(G) \ge 29$. The degree estimate 
$N\ge 2d+1$ implies that $N(G) \ge 29$ when the degree of $G$ is at least 14. Therefore 
it remains only to  check the case when $\text{deg}(G) =13$, or equivalently when $\text{deg}(H) =6$.

We proceed as before.
Define $H$ as in equation \eqref{eq:H}. Before the statement of the Theorem we have listed all the terms in the resulting $G = F-H +HF$.
Here $U$ and $V$ are the only coefficients that can be negative since all other coefficients appear alone as coefficients of $G$.

The case $V<0$: Three of the coefficients are $V+7K, V+7L,V+7M$. Thus $K, L, M$ must be positive. The first row gives 3 terms. Rows 8, 9, 10-15 give $24$ more terms. From $(7 R+7 S+7 T+7 V )$, at least one of $R, S, T$ must be positive. Suppose $R>0$, then the coefficients $R$ and $(14B+7R)$ must be positive. Thus $N(G)\ge3+ 24 +2 =  29$.

The case $V\geq 0$: The 3 possibilities are $U<0$, $U=0$, and $U>0$.
Furthermore, at least one of $K, L, M$ must be non-zero for $G$ to be degree 13. Without loss of generality, assume $K>0$. When $U>0$, we get $N(G)\geq 33$. 

When $U=0$, there are 17 terms with K that must appear in $G$. We also have the 3 constant terms. The $(14-U)$ coefficient yields another term in $G$. The coefficients in each column of the second and third rows cannot simultaneously vanish, and we get another 3 terms in $N(G)$. Similarly, rows 5 and 6 give another 3 and rows 10 and 11 give 2 more terms. Thus, $N(G)\geq 3+17+1+3+3+2 =29$.

When $U<0$, we have the initial 3 constant terms and 15 terms with $K$ and no $U$. Including also the $14-U$ coefficient gives us 19 positive coefficients so far. We pair the corresponding coefficients in the second and third rows, fifth and sixth rows to get an additional 6 terms. Pairing the 10th and 11th rows gives 2 more. Thus we have at least 27 terms. Finally using $U<0$ with the $U+7B+7L$ coefficient gives $B>0$ or $L>0$. Thus the $(7 B+D+14 L+7 R)$ and $\phi(0,C+B+S+T,L+V)$ coefficients must be positive. Again we have $N(G)\geq 29$.
\end{proof}

This theorem leaves undecided three possible values, namely $31, 35, 36$.  The authors believe these values are not possible.  
To find an example with $36$ terms, degree estimates force $G$ to be of degree at most $17$, and hence $H$ to be of degree at most $10$.
There are 39 (non-constant) monomials that are $\Gamma(7)$-invariant and degree at most 10. The resulting $G$ has 161 coefficients. The problem then is to determine how many of the coefficients of $G$ can be made to simultaneously vanish. The authors hope to address this computational problem in a future paper.


\section{Interpretation in terms of sparseness}

In applied mathematics it is often important to know how many components of an $n$-tuple of numbers are non-zero.  This number is often called its ${\ell}^0$ norm;
see \cite{IEEE} for its use in compressed sensing. See also \cite{DGL} for the connection with CR geometry. 
The ${\ell}^0$ norm $\|{\bf w}\|_0$ of a vector ${\bf w}$ is {\bf not} a norm; it is the number of 
non-vanishing components of ${\bf w}$ with respect to a fixed basis. 

Consider an affine map ${\bf v} \to W({\bf v})$ between vector spaces with bases chosen. It is natural to study the ${\ell}^0$ norm of $W({\bf v})$ as a function of ${\bf v}$,  but
this simple sounding question is annoyingly complicated.  The results in this paper can be interpreted in this manner. The set of coefficients in the charts in Section 5 can be regarded as images of an affine map.
We describe the easier situation for $\Gamma(2r+1,2)$. The proof of Theorem 1.2 for this group includes finding 
all special polynomials of degree at  most $4r+1$.  We show that such polynomials depend on
$r$ non-negative parameters. Let us denote these $r$ parameters by a vector ${\bf v}$ with respect to the standard basis in ${\mathbb R}^r$. 
This vector ${\bf v}$ lies in an explicit closed rectangle
in ${\mathbb R}^r$. We use its components as coefficients in defining an invariant polynomial $h_{\bf v}$.
All special polynomials of degree at most $4r+1$ then have the form
\begin{equation} W({\bf v})= f_{2r+1}(x,y) + h_{\bf v}(x,y) \left(f_{2r+1}(x,y) - 1\right). \label{eq:4} \end{equation}

Note that $W$ is an affine map. In \eqref{eq:4}, the coefficients of the polynomial $h_{\bf v}$ are the $r$ parameters. 
Theorem 1.2 for $ \Gamma(2r+1,2)$ has the following interpretation:
\begin{itemize}
\item The minimum ${\ell}^0$ norm satisfies $\|W({\bf v})\|_0=r+2$.
\item $\|W({\bf v})\|_0$ cannot lie in the integer interval $[1,r+1]$ nor in $[r+3, 2r+2]$.
\end{itemize}

Furthermore, if we allow polynomials of larger degree, we can regard $W$ as an affine map $h \mapsto W(h)$ from a vector space of invariant polynomials (of dimension larger than $r$) to itself.
Then
$$ W(h) = f _{2r+1} + (f_{2r+1}-1)h.$$
The last conclusion of Theorem 1.2 for this group then  states:
\begin{itemize}
\item For any $K \ge 2r+3$, there is an $h$ with $\|W(h)\|_0 = K$.  
\end{itemize}

Similar interpretations apply in the other cases; for $\Gamma(7)$ we have listed several explicit examples of this situation where the number of source variables and the dimension of the image of the affine map are large.

\begin{remark} Given an affine map from ${\mathbb R}^r$ to some ${\mathbb R}^K$, it seems to be quite difficult in general to describe the set of values of the $\ell^0$ norm. The situation in this paper involves 
an additional difficulty which we illustrate in the following example. \end{remark} 

\begin{example} The affine map here arises when considering Example 1.1 in source dimension $2$. Define $W:{\mathbb R}^3 \to {\mathbb R}^8$ by
$$ W(A,B,C) = (1-A, 2-B, 1-C, A, 2A+B, A+2B+C, B+2C, C). $$ 
The $\ell^0$ norm of $W(A,B,C)$ is generically equal to $8$. We note the following additional values and the corresponding invariant polynomial $p$. Each such polynomial equals $1$ when $x+y=1$.
\begin{itemize}
\item $\|W(0,0,0)\|_0 = 3$. Here $p(x,y) = (x+y)^2$.
\item$\|W(1,-2,1)\|_0  = 4$. Here $p(x,y) = 4xy+ x^4 - 2 x^2 y^2 + y^4$.
\item $\|W(1,2,1)\|_0 = 5$. Here $p(x,y) = (x+y)^4$. 
\end{itemize}

The invariant polynomial in the second item has a negative coefficient, and hence it is {\it not special}. There is no other $(A,B,C)$ for which $\|W(A,B,C)\|_0 = 4$. 
There are no $(A,B,C)$ for which $\|W(A,B,C)\|_0$ equals either $1$ or $2$. Thus $1,2$ are gaps in the possible values in general,
and $4$ is also a gap if we insist that $W(A,B,C)$ lies in the  orthant $Q$ where all the target variables are non-negative.
\end{example} 

We now summarize the situation from this perspective. Assume that the basic polynomial $F$ is known. We consider a collection ${\bf S}$ of polynomials $H$. (In this paper, $H$ must be group-invariant.)
We regard the coefficients of $H$ as parameters and form $G=F-H+HF$. When the degree of $H$ is bounded, the list of coefficients of $G$ can be thought of as defining an affine map from some real Euclidean space to one of typically larger dimension. Let us write $W(H)$ for the image of this map. The condition that $F$ has non-negative coefficients forces these coefficients to satisfy a system of linear inequalities. We want to describe the set of possible values of $N(G)$, equivalently $\|W(H)\|_0,$ as $H$ ranges over ${\bf S}$. 
For the group $\Gamma(2r+1,2)$, for example, the number of terms $N(G)$  can be either $r+2$ or any
positive integer at least $2r+3$, and no other values are possible. The result for $\Gamma(7)$ is considerably harder to state and appears in Theorem 5.1.

\vspace{-0.2cm}

\section{Concluding remarks}

The coefficients of special polynomials in our charts can be regarded as defining an affine map from ${\mathbb R}^N \to {\mathbb R}^K$, where $N$ is the number of parameters in our polynomial $H$ and $K$ is the number of coefficients. Here typically $K >>N$. These coefficients must be nonnegative and the structure of the equations leads to further constraints.
The result  is, for each degree $k$, a compact convex subset ${\bf S}_k$ of ${\mathbb R}^N$ in which the coefficients must lie. To analyze the possible values of $N(G)$ for $G \in {\bf S}_k$ one must study the boundary of ${\bf S}_k$.
Doing so is essentially equivalent to understanding sparseness for linear systems. We saw one simple example where the set consisted of the interval $[0,14]$. In this case $N(G_\lambda) =17$ for $\lambda=0$, $N(G_\lambda) =30$ for $0 < \lambda <1$, and $N(G_\lambda) = 29$ when $\lambda=14$. It is possible in principle to regard all of our calculations from this point of view. Perhaps the most important observation is that invariance reduces the size of the linear algebra problem in two ways.
First, the number of invariant monomials of up to a given degree is smaller than the number of all monomials up to that degree. Second, the invariance allows for a somewhat simplified manner of analyzing the boundary.

We next mention further  directions one could take along the lines of this paper.

The paper \cite{EX} studies the Bergman kernel function for the quotient of the ball by the first two classes of groups we have mentioned. It seems interesting to consider the case of the group in Example 1.3.

Since the possibilities for group-invariant rational sphere maps are so limited, it is natural to replace invariance with equivariance.
In \cite{DX1} and \cite{DX2}, one assigns a certain group $\Gamma_f$ to a rational sphere map $f$. This group consists of all automorphisms $\alpha$ of the source ball for which there exists an automorphism
$\beta$ of the target ball with $f\circ \alpha = \beta \circ f$. There are many possibilities for $\Gamma_f$. For example, given a finite group $G$, there always is a polynomial sphere map $f$ for which
$\Gamma_f$ is isomorphic to $G$. Furthermore, it is known precisely when $\Gamma_f$ is the full automorphism group, when it is the unitary group, when it contains the $n$-torus, and so on. See \cite{GWZ} and \cite{GL} for further results in this direction.
There remain combinatorial questions about the relationship between the degree of a polynomial map and the number of terms for a map with given group $\Gamma_f$.

Similar questions apply for the groups $\Gamma(p,q)$, but one must be cognizant of the warning preceding Example 2.1. When coefficients can be of both signs, it is extremely difficult to count the number
of terms in the product of two polynomials. It seems nonetheless worthwhile to study the gaps that arise in the number of possible terms for polynomials invariant under $\Gamma(p,q)$.


\begin{thebibliography}{82}

\medskip 

\bibitem{BG} J. Brooks, S. Curry, D. Grundmeier, P. Gupta, V.
Kunz, A. Malcom, and K. Palencia, Constructing
group-invariant CR mappings, 
Complex Anal. Synerg. 8 (2022), no. 4, Paper No. 20, 7 pp.

\medskip

\bibitem{C} S. Chiappari, Proper holomorphic mappings of positive codimension in several complex variables. PhD thesis, University of Illinois at Urbana-Champaign, 1990. 

\medskip

\bibitem{CD} D. Coppersmith and J. Davenport, Polynomials Whose Powers Are Sparse,  Acta Arith. 58 (1991), 79-87. 

\medskip

\bibitem{Drat} J. D'Angelo, Rational Sphere Maps, Progress in Mathematics, Volume 341, Birkh\"{a}user,  (2021).

\medskip

\bibitem{D2} J. D'Angelo, Number-theoretic properties of certain CR mappings. {\it J. Geom. Anal.}  14(2): (2004), 215-229.

\medskip

\bibitem{D3} J. D'Angelo, The combinatorics of certain group invariant maps. {\it Complex variables and elliptic equations},  Vol. 58,
Issue 5 (2013), 621-634. 


\medskip


\bibitem{DGL} J. D'Angelo, D. Grundmeier, and J. Lebl, Rational sphere maps, linear programming, and compressed sensing, {\it Complex Analysis and its Synergies} 6, 4 (2020). https://doi.org/10.1007/s40627-020-0041-5.


\medskip



\bibitem{DKR} J. D'Angelo, S. Kos, and E. Riehl, A sharp bound for the degree of proper monomial mappings between balls. J. Geom. Anal. 13 (2003), no. 4, 581-593. 

\medskip

\bibitem{DL} J. D'Angelo and D. Lichtblau, Spherical space forms, CR Maps, and proper maps between balls, {\it J. Geom. Anal.} 2 (1992), 391-415.

\medskip 

\bibitem{DX1} J. D'Angelo and M. Xiao, Symmetries in CR complexity theory, {\it Adv. Math.} 313 (2017), 590-627.

\medskip

\bibitem{DX2} J. D'Angelo and M. Xiao, Symmetries and regularity for holomorphic maps between balls, {\it Math. Res. Lett.}, Vol. 25, no. 5 (2018), 1389-1404.

\medskip

\bibitem{IEEE} D. Donoho, 
Compressed sensing, IEEE Trans. Inform. Theory 52 (2006), no. 4, 1289-1306.

\medskip


\bibitem{EX} P. Ebenfelt, M. Xiao, and H. Xu, On the classification of normal Stein spaces and finite ball quotients with Bergman-Einstein metrics, Int. Math. Res. Not. (IMRN), 2021.

\medskip

\bibitem{GWZ} E. Gevorgyan, H. Wang, A. Zimmer,  A rigidity result for proper holomorphic maps between balls, Proc. Amer. Math. Soc.152(2024), no.4, 1573–1585.

\medskip

\bibitem{G} D. Grundmeier,  Signature pairs for group-invariant Hermitian polynomials. Internat. J. Math. 22 (2011), no. 3, 311-343. 

\medskip

\bibitem{G2} D. Grundmeier, Group-invariant CR mappings. PHD Thesis,  University of Illinois at Urbana-Champaign (2011).

\medskip

\bibitem{GL} D. Grundmeier, J. Lebl, Rational Maps of Balls and their Associated Groups, to appear in Sao Paulo J. Math.

\medskip

\bibitem{GLW} D. Grundmeier,  K. Linsuain, B. Whitaker, Invariant CR mappings between hyperquadrics. Illinois J. Math.62(2018), no.1-4, 321–340.

\medskip

\bibitem{HJX}   X. Huang, S. Ji, and D. Xu, 
Several results for holomorphic mappings from ${\mathbb B}^n$ into ${\mathbb B}^N$, Geometric analysis of PDE and several complex variables, 267-292,
Contemp. Math., 368, Amer. Math. Soc., Providence, RI, 2005.

\medskip

\bibitem{HJY}  X. Huang, S. Ji, and W. Yin, On the third gap for proper holomorphic maps between balls, {\it Math. Ann. } 358 (2014), no. 1-2, 115-142.

\medskip

\bibitem{LP} J. Lebl and H. Peters, Polynomials constant on a hyperplane and CR maps of spheres, Illinois J. Math. 56 (2012), no. 1, 155-175 (2013).

\medskip


\bibitem{L} D.Lichtblau, Invariant proper holomorphic maps between balls. Indiana Univ. Math. J. 41 (1992), no. 1, 213-231. 

\medskip

\bibitem{LWW} N. Loehr, G. Warrington, and H. Wilf, The combinatorics of a three-line circulant determinant. Israel J. Math. 143 (2004), 141-156.


\end{thebibliography}
\end{document}